\documentclass[a4paper,10pt]{amsart}

\usepackage{textcomp}
\usepackage{lmodern}
\usepackage{amsmath,amsthm,amssymb}
\usepackage[T1]{fontenc}
\usepackage[english]{babel}
\usepackage{color} 
\usepackage[utf8]{inputenc}
\usepackage[hidelinks,pdfusetitle]{hyperref}
\usepackage[all]{xy}

\theoremstyle{definition}
\newtheorem{defn}{Definition}[section]
\newtheorem{remark}[defn]{Remark}

\newtheorem*{defn*}{Definition}
\theoremstyle{plain}
\newtheorem{thm}[defn]{Theorem}
\newtheorem{prop}[defn]{Proposition}
\newtheorem{lemma}[defn]{Lemma}
\newtheorem{corollary}[defn]{Corollary}

\newcommand{\numberset}{\mathbb}
\newcommand{\N}{\numberset{N}}
\newcommand{\Z}{\numberset{Z}}

\newcommand{\R}{\numberset{R}}
\newcommand{\C}{\numberset{C}}

\newcommand{\T}{\mathcal{T}}

\newcommand{\calD}{\mathcal{D}}
\newcommand{\F}{\mathcal{F}}
\newcommand{\calP}{\mathcal{P}}
\newcommand{\calQ}{\mathcal{Q}}
\newcommand{\D}{\mathcal{D}}

\newcommand{\calL}{\mathcal{L}}

\DeclareMathOperator{\End}{End}
\DeclareMathOperator{\Id}{Id}

\DeclareMathOperator{\cl}{cl}

\DeclareMathOperator{\ad}{ad}

\DeclareMathOperator{\Der}{Der}

\renewcommand{\epsilon}{\varepsilon}
\renewcommand{\theta}{\vartheta}
\renewcommand{\phi}{\varphi}

\title{Formal connections for families of star products}

\author{J\o rgen Ellegaard Andersen}
\address{Centre for Quantum Geometry of Moduli Spaces, Aarhus University, Ny Munkegade 118, bldg. 1530, DK-8000 Aarhus C Denmark}
\email{andersen@qgm.au.dk}

\author{Paolo Masulli}
\address{Centre for Quantum Geometry of Moduli Spaces, Aarhus University, Ny Munkegade 118, bldg. 1530, DK-8000 Aarhus C Denmark}
\email{paolo.masulli@gmail.com}

\author{Florian Sch\"atz}
\address{Centre for Quantum Geometry of Moduli Spaces, Aarhus University, Ny Munkegade 118, bldg. 1530, DK-8000 Aarhus C Denmark}
\email{florian.schaetz@gmail.com}

\thanks{
Partially supported by the center of excellence grant ’Centre for Quantum Geometry of Moduli Spaces’ from the Danish National Research Foundation (DNRF95).}

\begin{document}

\begin{abstract}
We define the notion of a formal connection for a smooth family of star products with fixed underlying symplectic structure.
Such a formal connection allows one to relate star products at different points in the family.
This generalizes the formal Hitchin connection defined in \cite{andersen2012hitchin}.
 We establish a necessary and sufficient condition that guarantees the existence of a formal connection, and we describe the space of formal connections for a family as an affine space modelled on the formal symplectic vector fields. Moreover
 we show that if the parameter space has trivial first cohomology group, any two flat formal connections are related by an automorphism
 of the family of star products.
\end{abstract}

\maketitle

\section{Introduction}\label{section: intro}

\subsection{Quantization and the Hitchin connection}

In his seminal paper \cite{witten1989quantum} Witten investigated quantum Chern-Simons theory,
a $3$-dimensional topological quantum field theory (TQFT). As the $2$-dimensional part 
of this theory he proposed
the geometric quantization
of the moduli space of flat connections on a Riemann surface $\Sigma$.
 This moduli space has a natural symplectic structure $\omega$ and admits a prequantum line bundle, i.e. a Hermitian line bundle $\calL$ with a compatible connection, whose curvature is given by the symplectic form.

The Teichm\"{u}ller space $\T$ of the surface $\Sigma$ parametrizes complex structures on the
moduli space, so for each point $\sigma \in \T$ and each natural number $k$, called the level
of quantization, we have the quantum state space of geometric quantization, which is the space 
\[
Q_k(\sigma) = H^0(M_\sigma; \calL^{k} )
\] 
of holomorphic sections of the $k$-th tensor power of the prequantum line bundle. These form the fibres of a vector bundle $Q$ over $\T$, called the Verlinde bundle, and it was shown independently by Hitchin \cite{hitchin1990flat} and Axelrod, Della Pietra and Witten \cite{axelrod1991} that this bundle admits a natural projectively flat connection, which we shall call the \emph{Hitchin connection} (see also \cite{andersen2012hitchin} for a purely finite dimensional differential geometric approach to this connection). Consequently, the quantum spaces associated with different complex structures are identified, as projective spaces, through the parallel transport of this connection. By the work of Lazslo combined with the work of the first author and Ueno \cite{L}, \cite{AU1,AU2,AU3,AU4}, this provides a geometric construction of the vector spaces the Witten-Reshetikhin-Turaev TQFT associates to a closed oriented surface \cite{RT1,RT2,T}.

\subsection{Formal connections}

On a Poisson manifold $M$, a \emph{deformation quantization}, or \emph{star product}, is a $\C[[h]]$-linear product on the space $C^\infty(M)[[h]]$ that is associative, reduces to the pointwise product modulo $h$, and such that the component of degree $1$ in $h$ of its commutator is the Poisson bracket.

In the symplectic case, the existence of star products was originally established by De Wilde and Lecomte in \cite{de1983existence} via a cohomological approach. Fedosov provided a more direct and geometrical construction of star products in \cite{fedosov1994simple}. In the general Poisson case, the question of existence and classification was settled by Kontsevich, as a consequence of his Formality theorem \cite{kontsevich2003deformation}.

The first named author of this paper has studied the asymptotic relationship between Toeplitz operators and the Hitchin connection in \cite{andersen2006asymptotic}, and extended this asymptotic analysis to higher orders \cite{andersen2012hitchin}, which led him to define the following notion.

\begin{defn}
Let $M$ be a symplectic manifold equipped with a smooth family of star products $\{\star_\sigma\}_{\sigma \in \T}$ parametrized by a manifold $\T$. A \emph{formal connection} for
$\{\star_\sigma\}_{\sigma\in \T}$ is a connection on the bundle $\T \times C^\infty(M)[[h]] \to \T$ of the form:
\begin{equation}\label{eq:introduction_formal_connection}
D_V f = V[f] + A(V)(f),
\end{equation}
where $A$ is a smooth 1-form on $\T$ with values in differential operators on $M$ such that $A = 0 \, \mathrm{mod}\, h$, $f$ is a smooth section of the bundle, $V$ is any smooth vector field on $\T$, and $V[f]$ denotes the derivative of $f$ along $V$. 

The formal connection is called \emph{compatible} with the family of star products $\{\star_\sigma\}_{\sigma\in\T}$, if it is a derivation of the products for any vector field $V$ on $\T$, i.e.
\begin{equation}\label{eq:introduction_formal_connection_is_derivation}
D_V(f \star_\sigma g) = D_V(f) \star_\sigma g + f \star_\sigma D_V( g),
\end{equation}
for all $\sigma \in \T$ and all smooth sections $f$ and $g$ of $\T \times C^\infty(M)[[h]] \to \T$.
\end{defn}

Recall that on a K\"ahler manifold $M$ one can consider the Berezin-Toeplitz star product $\star^{BT}$, which can be constructed via the theory of Toeplitz operators, see \cite{schlichenmaier1999deformation,schlichenmaier2011deformation}.\footnote{
One of the earliest studies of star products on K\"ahler manifolds was Berezin's work \cite{berezin1974quantization}. Berezin's approach was extended to arbitrary K\"ahler manifolds by Reshetikhin and Takhtajan in \cite{reshetikhin2000deformation}.}
This leads us to consider the following class of examples of families of star products:

\begin{defn}
Let $M$ be a symplectic manifold with a family of compatible almost complex structures parametrized by a complex manifold $\T$, so that for any $\sigma \in \T$, the manifold $M_\sigma$ is a K\"ahler manifold, and let $\{\star_\sigma^{BT}\}_{\sigma\in \T}$ be the associated family of Berezin-Toeplitz star products. A \emph{formal Hitchin connection} on $M$ is a formal connection that is compatible with this family of star products and that is flat.
\end{defn}

In \cite{andersen2012hitchin} a particular formal Hitchin connection that is associated to the Hitchin connection from geometric quantization was studied, and it was shown that the projective flatness of the Hitchin connection implies the flatness of this formal Hitchin connection. An explicit expression for the $1$-form $\tilde{A}(V)$
for the formal Hitchin connection associated to the Hitchin connection was given in \cite{andersen2012hitchin,andersen2011hitchin}. The formula reads
\begin{equation}\label{eq:explicit_formula_FHC_introduction}
\tilde{A}(V)(f) = -V[F]f + V[F] \star^{BT} f + h(E(V)(f) - H(V)\star^{BT} f).
\end{equation}
where $E$ is a $1$-form on $\T$ with values in differential operators on $M$, $H$ is the $1$-form with values in $C^\infty(M)$ given by $H(V) = E(V)(1)$, and $F$ is the Ricci potential of the family. The construction will be recalled in more detail in Section \ref{section: Hitchin connection}.
 
It was furthermore noticed in \cite{andersen2012hitchin}, that -- provided certain cohomology groups of the mapping class group vanish -- any formal Hitchin connection could be used to obtain a mapping class group equivariant deformation quantization on the moduli space. The first author has further applied these constructions to the WRT-TQFT using the geometric description of the TQFT vector spaces described above in \cite{A7, A8}. See also \cite{A5} and \cite{andersenblaavand2011}, where the first author gave an explicit expression for the parallel transport of formal Hitchin connection in the abelian case.

With this motivation in mind, we return to the general case.

\subsection{Existence and classification of formal connections}

Let $(M,\omega)$ be a symplectic manifold with a family of natural star products $\{\star_\sigma\}_{\sigma\in\T}$ that is parametrized by $\T$.

We first establish the following necessary and sufficient condition for the existence of a compatible formal connection:

\begin{thm}\label{thm:introduction_existence_formal_connection}
 There exists a formal connection compatible with $\{\star_\sigma\}_{\sigma\in \T}$ if and only if the characteristic class 
 $[\cl(\star_\sigma)]$
 of the star products is locally constant on $\T$.
\end{thm}

In fact, we show that even more is true: we provide a natural way to associate a formal connection to every trivialization for the family of characteristic $2$-forms
associated to the family of star products and establish a formula for its curvature.
The proof relies on Fedosov's geometrical construction of star products \cite{fedosov1994simple}.

This result specializes to the case of a symplectic manifold equipped with a smooth family of compatible K\"ahler structures, where we take the family of Berezin-Toeplitz star products associated with them. In this situation 
Theorem \ref{thm:introduction_existence_formal_connection} yields the existence
of a compatible formal connection.

\begin{thm}
Let $(M,\omega)$ be a compact, symplectic manifold, and let $\T$ be a complex manifold parametrizing a family of compatible K\"ahler structures $I_\sigma$ on $M$, with $\sigma \in \T$. The family of Berezin-Toeplitz star products associated with the family has constant characteristic class, and therefore admits a compatible formal connection.
\end{thm}

Let us assume in the following that a family of star products $\{\star_\sigma\}_{\sigma \in \T}$ admits
a compatible formal connection. We want to understand the space of all such formal connections.
Relying on a result of Gutt and Rawnsley, \cite{gutt1999equivalence}, we obtain:
%To this end, we denote the space of the derivations of the star product $\star_\sigma$ which are trivial modulo $h$ by $\Der_0(M,\star_\sigma)$.
%We shall see that the formal connections compatible with the family of star products form an affine space over the space of $1$-forms on $\T$ with values in $\Der_0(M,\star_\sigma)$. 

\begin{thm}\label{thm:introduction_affine_space_formal_connections}
Let $M$ be a symplectic manifold equipped with a smooth family of star products $\{\star_\sigma\}_{\sigma\in\T}$ parametrized by $\T$. The space $\F(M,\star_\sigma)$ of formal connections on $M$ that are compatible with the family of star products is an affine space over the space of $1$-forms on $\T$ with values in formal symplectic vector fields on $M$, and it can then be written as:
%\[
%\F(M,\star_\sigma) = D_0 + \Omega^1(\T,\Der_0(M,\star)),
%\]
\[
\F(M,\star_\sigma) = D_0 + \Omega^1(\T,h\Gamma_\mathrm{sym}(M)[[h]])
\]
for a fixed formal connection $D_0$, where $\Gamma_\mathrm{sym}(M)$ denotes the space of symplectic vector fields on $M$.
\end{thm}

Furthermore, if we assume that $H^1(M;\R)$ vanishes, all symplectic vector fields are Hamiltonian, and therefore all derivations of a star product on $M$ are essentially inner, therefore they are parametrized by elements of $\tilde{C}_h^\infty(M)$, the space of formal functions on $M$ modulo constants.

\subsection{Gauge transformations of formal connections}

We study the action of gauge transformations on the space of formal connections $\F(M,\star_\sigma)$. The transformations which we consider are differential self-equivalences of the family of star products, since the gauge transformations
have to preserve the compatibility with the family of star products.
If we assume that the parameter space $\T$ of the family of star products has trivial first cohomology group, i.e. $H^1(\T,\mathbb{R})=0$, we obtain the following result.

\begin{thm}\label{thm:gauge_transformation_of_flat_formal_connections_introduction}
Let $M$ be a symplectic manifold with a family of star products $\{\star_\sigma\}_{\sigma\in\T}$ parametrized by a smooth manifold $\T$ with trivial first cohomology group. Let $D,D' \in \F(M,\star_\sigma)$ be formal connections for the family and let us assume that they are flat. Then they are gauge equivalent via a self-equivalence of the family of star products $P \in C^\infty(\T,\D_h(M))$, meaning that 
\begin{equation}\label{eq:gauge_equivalence_formal_connections_introduction}
D_V' = P^{-1}D_VP,
\end{equation}
for any vector field $V$ on $\T$.
\end{thm}

This implies the following corollary:

\begin{corollary}
Let $\T$ be a smooth manifold with trivial first cohomology group, i.e. $H^1(\T,\mathbb{R})=0$. If there exists a formal Hitchin connection $D$ in the bundle $\T \times C^\infty(M)[[h]]$ over $\T$, then it is unique up to gauge-equivalence.
\end{corollary}

\subsection{The formal Hitchin connection at low orders}
As mentioned above, a formal connection is called compatible with a family of star products if it is a derivation with respect to the star products. 
Our main example of a formal connection is the formal Hitchin connection defined in \cite{andersen2012hitchin},
which is known to be a derivation with respect to the family of Berezin-Toeplitz star products,
\cite{andersen2012hitchin}. This fact relies on the existence of the Hitchin connection in geometric quantization, as well as the link between geometric and deformation quantization via Toeplitz operators.

This result assumes the existence of a Hitchin connection in geometric quantization, which puts several constraints on the objects involved in the construction, among them the condition that the family of K\"ahler structures be \emph{holomorphic} and \emph{rigid}, which are quite strong requirements. 

On the other hand, the explicit expression \eqref{eq:explicit_formula_FHC_introduction} that was obtained in \cite{andersen2012hitchin} makes sense in a more general situation, and therefore we can ask whether that expression in general gives a derivation of the Berezin-Toeplitz star product. This is a difficult question, because it involves the coefficients of the star product, which are in general hard to understand. But we can give an affirmative answer if we restrict our attention to the first order in the formal parameter. 
Moreover it is easy to check that the expression \eqref{eq:explicit_formula_FHC_introduction} defines a flat formal connection up to first order, and we conclude that the expression obtained in \cite{andersen2012hitchin} gives a formal Hitchin connection up to order one.

\begin{prop}
Let $M$ be a symplectic manifold with a family of compatible K\"ahler structures parametrized by a complex manifold $\T$. Then the expression \eqref{eq:explicit_formula_FHC_introduction} defines a formal connection that, modulo $h^2$, is a derivation of the family of Berezin-Toeplitz star products on $M$ and flat. Therefore it defines a formal Hitchin connection in the sense of our definition above, modulo terms of order $h^2$.
\end{prop}

We plan to investigate the relationship between the formal Hitchin connection studied by Andersen
and the formal connections obtained from Theorem \ref{thm:introduction_existence_formal_connection}
more closely
in the future.

\section{Deformation quantization}\label{section: deformation quantization}

\subsection{Star products}

Let $M$ be a Poisson manifold. We let $\C[[h]]$ denote the ring of formal power series with complex coefficients, and $C_h^\infty(M) = C^\infty(M)[[h]]$ the algebra of \emph{formal functions} on $M$, which are formal power series with coefficients in $C^\infty(M)$. Then $C_h^\infty(M)$ is an algebra over $\C[[h]]$, and we can extend the Poisson bracket linearly to make $C_h^\infty(M)$ into a Poisson algebra. This allows us to formulate the following definition.

\begin{defn}\label{defn:star_product}\index{star product|textbf}
Let $(M,\{\cdot,\cdot\})$ be a Poisson manifold. A (formal) \emph{star product} on (or \emph{deformation quantization} of)  $M$ is a $\C[[h]]$-bilinear map 
$$\star \colon C_h^\infty(M) \times C_h^\infty(M) \to C_h^\infty(M)$$ written as
\[
f \star g = \sum_{k=0}^\infty c^k(f,g)h^k,
\]
where, for $k\in \N$, the maps $c^k \colon C^\infty(M) \times C^\infty(M) \to C^\infty(M)$ are bilinear and called the \emph{coefficients} of $\star$. A star product is required to satisfy the following conditions:
\begin{enumerate}
\item associativity: $(f_1 \star f_2) \star f_3 = f_1 \star (f_2 \star f_3)$,
\item unitality: $f \star 1 = f = 1 \star f$,
\item $c^0(f_1,f_2) = f_1f_2$,
\item $c^1(f_1,f_2) - c^1(f_2,f_1) = i\{f_1,f_2\}$,
\end{enumerate}
for all $f_1,f_2,f_3 \in C^\infty(M)$.
\end{defn}

\index{star product!differential}\index{star product!null on constants}
A star product is said to be \emph{differential} if the coefficients $c^k$ are bidifferential operators, in the sense that, for a fixed $f \in C^\infty(M)$, both $c^k(f,\cdot)$ and $c^k(\cdot,f)$ are differential operators for all $k \in \N$.

\begin{defn}\index{star product!equivalence}
Two star products $\star, \star'$ on $M$ are said to be \emph{equivalent} if there is a formal power series of linear maps
\[
T = \sum_{k=0}^\infty T_k, \qquad T_k \colon C_h^\infty(M) \to C_h^\infty(M), \quad \text{$k \in \N$},
\] 
such that $T_0 = \Id$ and $T(f_1) \star' T(f_2) = T(f_1 \star f_2)$, for $f_1,f_2$ smooth functions on $M$.
\end{defn}

In the following, we will restrict attention to {\em natural star products}. These are differential star products
that satisfy the requirement that the coefficient $c^k$ is a bidifferential operator of order at most $k$.
All star products which we will encounter are of this type.

\subsection{Fedosov star products}\label{subsection: Fedosov}

We review the main ingredients of Fedosov's geometric construction
of star products \cite{fedosov1994simple}. In Section \ref{section: existence}, we will make
use of Fedovo's framework to construct formal connections.
Our exposition follows \cite{fedosov1994simple} and Waldmann's book \cite{waldmann2007poisson}.

Let $(M,\omega)$ be a symplectic manifold of dimension $m=2n$. Since any tangent space $T_x M$
is a symplectic vector space, we can consider the associated Weyl algebra.
\begin{defn}\label{def:weyl_algebra}\index{Weyl algebra}
The \emph{formal Weyl algebra} $W_x$ associated to $T_x M$, for $x \in M$, is the associative $\C$-algebra with unit, whose elements are formal power series in $h$, with formal power series on $T_x M$ as coefficients. This means that an element in $W_x$ has the form:
\[
a(y,h) = \sum_{k\in\N}\sum_{\alpha} h^k a_{k,\alpha} y^\alpha,
\]
where $(y^1,\dots,y^{m})$ are local coordinates on $T_x M$ and $\alpha = (\alpha_1,\dots,\alpha_m)$ is a multi-index. 

\end{defn}
The formal Weyl algebra is equipped with the following \emph{Moyal-Weyl product}:
\begin{equation}\label{eq:moyal_weyl_product}
a \circ_{MW} b = \sum_{k=0}^\infty \left( \frac{ih}{2} \right)^k \frac{1}{k!} \pi^{i_1j_1}\cdots\pi^{i_kj_k} \frac{\partial^k a}{\partial y^{i_1} \cdots \partial y^{i_k}}\frac{\partial^k b}{\partial y^{i_1} \cdots \partial y^{i_k}},
\end{equation}
where $\pi=\sum_{i<j}\pi^{ij}\frac{\partial}{\partial y^i}\wedge \frac{\partial}{\partial y^j}$ is the Poisson
bivector dual to $\omega_x$.
The Moyal-Weyl product is a deformation quantization of the linear symplectic space $(T_xM,\omega_x)$.

Let $W = \cup_{x\in M} W_x$. This defines a bundle of algebras over $M$, which is called the \emph{Weyl bundle}. The space of smooth sections of this bundle, $\Gamma W$, gives an associative algebra with fibre-wise multiplication. This space of section can be thought of as a ``quantized tangent bundle'' of $M$.
Note that the centre of $\Gamma W$ is formed by the elements that do not contain any $y^i$, and therefore is naturally identified with $C_h^\infty(M)=C^\infty(M)[[h]]$.

We next define a useful degree on the Weyl algebra, called the {\em total degree}. It is given
by assigning degree $1$ to all the $y^i$,  i.e. $\deg y^i = 1$ for any $i$, and $\deg h = 2$.
The Moyal-Weyl product is additive with respect to the total degree.

A differential form on $M$ with values in $W$ is a section of $W\otimes \Lambda^q T^* M$, and can be expressed as:
\[
a(x,y,h,dx) = \sum h^k a_{k,i_1,\dots,i_p,j_1,\dots,j_q} y^{i_1} \dots y^{i_p} dx^{j_1} \wedge \dots \wedge dx^{j_q}
\]
in local coordinates, where the coefficients $a_{k,i_1,\dots,i_p,j_1,\dots,j_q}$ are symmetric in the $i$'s and anti-symmetric in the $j$'s. We denote the space of these differential forms by $\Omega(M,W)$.

Note that the Moyal-Weyl product on $\Gamma W$ extends to $\Omega(M,W)$.
On the latter space
we define an operator $\delta$ in the following way:
\[
\delta(a) = \sum_i dx^i \wedge \frac{\partial a}{\partial y^i }, \qquad \text{for all $a \in W\otimes \Lambda^q T^* M$.}
\]
This operator can also be written as
\[
\delta(a) = - \left[ \frac{i}{h} \omega_{ij}dy^i dx^j,a \right] = -\frac{i}{h}\ad(\tilde{\omega}),
\]
where the commutator is with respect to the Moyal-Weyl product
and $\tilde{\omega}$ is $\omega$, seen as a section of $T^*M\otimes T^*M \subset W\otimes \wedge T^*M$.

An alternative interpretation of $\delta$ is as follows: we can define two commuting derivations $\delta$ and $\delta^*$ 
on $\Omega(M,W)$
by considering the identity morphism $TM\to TM$ as a section of $TM\otimes T^*M$. We can now insert the $TM$-part into either $\Lambda T^*M$ or $W$ and multiply the $T^*M$-part with the other factor, using the pointwise product on $W$.
The operator $\delta$ corresponds to the latter case. We define $\delta^*$ to be the operator corresponding to the former.

%\nrt{
The operators $\delta$ and $\delta^*$ are differentials, i.e. they square to zero. Moreover, one can use $\delta^*$ as a homotopy
operator for $\delta$, which leads to the result that the cohomology of $\delta$ is concentrated in form-degree $0$
and that $H^0(\Omega(M,W),\delta)=C_h^\infty(M)$.
%}

\begin{defn}\index{connection!symplectic}
A \emph{symplectic connection} on a symplectic manifold $(M,\omega)$ is a linear connection $\nabla$ that is torsion-free and such that $\omega$ is parallel with respect to $\nabla$, i.e. $\nabla \omega = 0$.
\end{defn}

Let us fix a symplectic connection $\nabla$ on $M$. Its curvature tensor is contracted with $\omega$ to yield an element $R \in \Omega^2(M,W)$.
As usual, $\nabla$ extends to the covariant derivative $d_{\nabla}$ on $\Omega(M,TM)$,
which we dualize and extend to $\Omega(M,W)$. It turns out that $d_{\nabla}$
is also a derivation for the fibre-wise Moyal-Weyl product.
Moreover, $d_{\nabla}$ commutes with $\delta$ and squares to $-\frac{i}{h}\ad(R)$.

Fedosov's idea is to correct the non-flatness of $-\delta + d_{\nabla}$ by finding an appropriate $r\in \Omega^1(M,W)$ such that the total operator
\begin{equation}\label{eq:D_weyl_connection}
D_r := -\delta + d_{\nabla} + \frac{i}{h}\ad(r)
\end{equation}
squares to zero. Regardless of flatness, an operator of the form of $D_r$ is a derivation of the Moyal-Weyl product defined above, i.e.
\[
D_r(a \circ_{MW} b) = D_r(a) \circ_{MW} b + a \circ_{MW} D_r(b)
\]
holds.
Using Fedosov's ansatz for $D_r$, one computes
\[
D_r^2 = \frac{i}{h}\ad\left(-\omega - \delta r + R + d_{\nabla}r + \frac{i}{h} r\circ_{MW}r \right).
\]

Hence the flatness of $D_r$ is equivalent to the fact that
\begin{equation}\label{eq:scalar_form_Weyl_connection}
\alpha = \omega + \delta r - R - d_{\nabla}r - \frac{i}{h} r\circ_{MW}r
\end{equation}
lies in the center of $\Omega(M,W)$, which coincides with $\Omega(M)[[h]]$, where we see ordinary
differential forms to be fibrewise constant polynomials on the tangent spaces. 

We call a connection of the form \eqref{eq:D_weyl_connection} \emph{abelian} if $\alpha$ satisfies this condition, i.e. if it is a scalar $2$-form. If this is the case, then by the Bianchi identity we have that $d\alpha = D_r\alpha = 0$, and so $\alpha$ is closed (i.e. $\alpha \in Z^2(M)[[h]]$) and is called the \emph{Weyl curvature}\index{Weyl curvature} of $D_r$.

The following theorem shows how to construct the appropriate $r$ in order to obtain an abelian $D_r$.

\begin{thm}{(Fedosov)}\label{thm:fedosov_char_class}
Let $\nabla$ be a symplectic connection on $M$ and 
\[
\alpha = \omega + h\alpha_1 + h^2\alpha_2 + \dots \in Z^2(M)[[h]]
\]
be a closed formal $2$-form that is a perturbation of the symplectic form. Then there exists a unique $r \in \Omega^2(M,W)$, such that the $D_r$ given by \eqref{eq:D_weyl_connection} is an abelian connection with Weyl curvature $\alpha$ and $\delta^{*} r = 0$.
\end{thm}

We are now in position to define the Fedosov star product corresponding to a given Fedosov connection.
One first shows that every element $f\in C^\infty(M)[[h]]$ extends uniquly to an element 
\[
\tau(f) \in \Omega^0(M,W)[[h]]
\]
which is parallel with respect to the Fedosov connection $D_r$.
The proof amounts to breaking up the equation $D_r\tau(f) = 0$
into its homogeneous pieces with respect to the total degree. Cohomological considerations
similar to the ones in the proof of Theorem \ref{thm:fedosov_char_class}
guarantee that the extension $\tau(f)$ exists and is unique.

The \emph{Fedosov star product} $\star_{\nabla,\alpha}$ associated to the closed formal $2$-form $\alpha$ and to the symplectic connection $\nabla$ on $M$ is given by the following formula:
\begin{equation}\label{eq:fedosov_star_product}
f\star_{\nabla,\alpha} g := p(\tau(f) \circ_{MW} \tau(g)),
\end{equation}
where $p: \Gamma(W) \to C_h^\infty(M)$ is the projection of a section
of the Weyl bundle to its fibrewise constant part. 
Since $D_r$ is a derivation with respect to $\circ_{MW}$, this defines an associative product
and one checks inductively that one actually obtains a natural star product, i.e.
$\star_{\nabla,\alpha}$ is given by bidifferential operators whose component of order $h^k$ is of differential-order at most $k$ in each argument.

\begin{defn}
The \emph{characteristic $2$-form} of a Fedosov star product $\star_{\nabla,\alpha}$ is
\[
\cl (\star_{\nabla,\alpha}) =  \alpha \in Z^2(M;\R)[[h]],
\]
where $\alpha$ is the Weyl curvature of the corresponding Fedosov connection $D_r$.

The class of $\cl(\star_{\nabla,\alpha})$ is called the characteristic class of $\star_{\nabla,\alpha}$.
\end{defn}

%\nrt{
The definition of the characteristic class of Fedosov star products generalizes to arbitrary
differential star products since
every differential star product is equivalent to 
one of Fedosov type
and two Fedosov star products are equivalent if and only if their characteristic classes coincide.
Hence one just defines the characteristic class of an arbitrary differential star product
to be the characteristic class of an equivalent star product of Fedosov type.
For a more extensive treatment of classification results, we refer the reader to \cite{deligne1995deformations}, \cite{gutt1999equivalence} as well as the exposition in \cite{waldmann2007poisson}.

We note that if we restrict attention to natural star products,
a stronger statement holds, see  \cite[Theorem 4.1]{gutt2003natural}:
every such star product is equivalent to a {\em preferred} Fedosov star product
through a {\em preferred} equivalence.
This allows one to assign not only a characteristic class, but a characteristic $2$-form
to a natural star product -- just take it to be the characeristic $2$-form of the preferred equivalent Fedosov star product.
%}

\subsection{The Berezin-Toeplitz star product}\label{subsection: BT}

We describe the link between geometric quantization and deformation quantization due to \cite{schlichenmaier1999deformation}.

\begin{defn}\index{line bundle!prequantum}
A \emph{prequantum line bundle} over a symplectic manifold $(M,\omega)$ is the data of a complex line bundle $\calL$, equipped with a Hermitian metric $h$ and a compatible connection $\nabla$ whose curvature satisfies:
\[
F_{\nabla} = -i \omega.
\]
\end{defn}
We say that a symplectic manifold is \emph{prequantizable} if it admits a prequantum line bundle. If we assume that $M$ is a compact prequantizable K\"ahler manifold we can make the following definition.

\begin{defn}\index{Toeplitz operator|textbf}
Let $f \in C^\infty(M)$. The \emph{Toeplitz operator} $T_f^{(k)} \colon C^\infty(M;\calL^k) \to H^0(M;\calL^k)$ is the map defined by
\[
T_f^{(k)}(s) = \pi^{(k)}(fs),
\]
mapping a smooth section $s$ on $\calL^k$ to its projection onto the subspace of holomorphic sections.
\end{defn}

Schlichenmaier showed in  \cite{schlichenmaier1999deformation} that any compact K\"ahler manifold admits a natural star product on it, namely the Berezin-Toeplitz star product.

\begin{thm}[Schlichenmaier]\label{thm:BT_star_product}\index{star product!Berezin-Toeplitz}
There exists a unique star product $\star^{BT}$ for $M$, called the \emph{Berezin-Toeplitz star product}, which is expressed by:
\[
f_1 \star^{BT} f_2 = \sum_{k=0}^\infty c^{(k)}(f_1,f_2)h^k,
\]
with $c^{(k)}(f_1,f_2) \in C^\infty(M)$ determined by the requirement that for all $f_1,f_2 \in C^\infty(M)$ and for any positive integer $L$ the following estimate holds:
\[
\Bigg\lVert T_{f_1,\sigma}^{(k)} T_{f_2,\sigma}^{(k)} - \sum_{l=0}^L T_{c_\sigma^{(l)}(f_1,f_2),\sigma}^{(k)} k^{-l} \Bigg\rVert = O(k^{-(L+1)}).
\]
\end{thm}

Karabegov and Schlichenmaier proved \cite{karabegov2001identification} that the Berezin-Toeplitz star product is of \emph{Wick type}. 
This means that for any locally defined functions $f$ and $g$ with $f$ anti-holomorphic and $g$ holomorphic, and any function $h$ one has
$$f\star h = fh \quad \textrm{and} \quad h\star g = hg.$$

\begin{remark}
A star product is said to be \emph{with separation of variables} if it has the same property characterizing star products of Wick type, but with the roles of holomorphic and anti-holomorphic switched. For any star product $\star$, we can define the \emph{opposite} star product $\star_o$ by setting $f \star_o g := g \star f$. Therefore, if $\star$ is a of Wick type, then $\star_o$ is with separation of variables.
\end{remark}

Gammelgaard \cite{gammelgaard2010universal} showed that star products with separation of variables can be expressed locally in a graph theoretical way, with weights determined by the automorphisms of the graphs.

The formal Hitchin connection that we define in the next section is closely related to the Berezin-Toeplitz star product.

\section{The Hitchin connection}\label{section: Hitchin connection}

Here we review briefly the construction of the Hitchin connection, in the differential geometric version of  \cite{andersen2012hitchin}. The Hitchin connection was introduced by Hitchin in \cite{hitchin1990flat} as a connection over the Teichm\"uller space in the bundle one obtains by applying geometric quantization to the moduli spaces of flat $SU(n)$ connections. Furthermore Hitchin proved that this connection is projectively flat. Hitchin's construction was motivated by Witten's study \cite{witten1989quantum} of quantum Chern-Simons theory in $2 + 1$ dimensions. In \cite{andersen2012hitchin} a differential geometric construction of the Hitchin connection which works for a more general class of manifolds was provided.

\subsection{Smooth families of K\"ahler structures}
Let $(M,\omega)$ be a symplectic manifold and let $\T$ be a smooth manifold that parametrizes smoothly a family of K\"{a}hler structures on $M$. This means that we have a smooth map 
\[
I \colon \T \to C^\infty(M, \End(TM))
\]
that associates to each $\sigma \in \T$ an integrable and compatible almost complex structure on $M$. The requirement that the map $I$ is smooth means that it defines a smooth section of the pullback bundle
\[
\pi_M^*(\End(TM)) \to \T \times M,
\]
where $\pi_M \colon \T \times M \to M$ denotes the canonical projection map.

The symplectic form $\omega$ is non-degenerate, therefore we obtain from it an isomorphism
$i_\omega \colon TM_\C \to TM^*_\C$
 by contraction in the first entry.
% or because it is of type $(1,1)$.
We can use this isomorphism to define the bivector field:
\[
 \tilde\omega = - (i_\omega^{-1} \otimes i_\omega^{-1})(\omega),
\] 
which satisfies the identity $\omega \cdot \tilde\omega = \tilde\omega \cdot \omega = \Id$, where the dot indicates contraction of tensors in their entries closest to the dot, which is relevant when working with non-symmetric tensors. For example $\omega \cdot \tilde\omega$ means that the right-most entry of $\omega$ is contracted with the left-most one of $\tilde\omega$.

Similarly we obtain a type-interchanging isomorphism $i_{g_\sigma} \colon TM_\C \to TM^*_\C$,
induced by the K\"{a}hler metric on $M_\sigma$. The two isomorphisms are related by the equation: $i_{g_\sigma} = I_\sigma i_\omega$. From the fact that $g$ and $\omega$ have type $(1,1)$ it follows that these two isomorphisms exchange types. As done for $\omega$, we can define the inverse metric tensor by
\[
\tilde g = (i_g^{-1} \otimes i_g^{-1}) (g) = \Id,
\]
which gives a symmetric bivector field satisfying the relation $g \cdot \tilde g = \tilde g \cdot g$. This bivector field is related to the bivector field associated to $\omega$ by
$\tilde \omega = I \cdot \tilde g$.

On an (almost) complex manifold $M$ we have a natural decomposition of the complexified tangent bundle:
\[
TM_\C = T'M_I \oplus T''M_I,
\]
where the two summands are the eigenspaces of the endomorphism $I$ for the eigenvalues $i$ and $-i$, respectively:
\[
T'M_I = \ker(I-i\Id), \qquad T''M_I = \ker(I+i\Id).
\]
Sections of the first subspace are said to be vector fields \emph{of type $(1,0)$}, and sections of the second subspace are vector fields \emph{of type $(0,1)$}.
The decomposition is explicitly given by the projections to the two subspaces:
\[
\pi_I^{1,0} = \frac{1}{2}(\Id -iI), \qquad \pi_I^{0,1} = \frac{1}{2}(\Id +iI),
\]
and for a vector field $X$ we denote its decomposition by $X = X_I' + X_I''$.

\subsection{A symmetric bivector field}

We now assume that $(M,\omega)$ is a symplectic manifold equipped with a smooth family of compatible almost complex structures $I$, parametrized by $\T$. We can define a bivector field $\tilde{G}(V) \in C^\infty (M,TM_\C \otimes TM_\C)$ by requiring that the relation
\begin{equation}\label{eq:definition_of_G(V)}
V[I] = (\Id \otimes \ i_\omega)(\tilde{G}(V))
\end{equation}
holds for all vector fields $V$. If we differentiate the identity $\tilde{g} = - I \cdot \tilde{\omega}$ along a vector field $V$ on $\T$, we get that
\[
V[\tilde{g}] = - V[I] \cdot \tilde{\omega} = - \tilde{G} (V),
\]
and since $\tilde{g}$ is a symmetric bivector field, so is $\tilde{G}(V)$. Moreover, because of the types of $V[I]$ and $\tilde{\omega}$, we get a decomposition
\[
\tilde{G}(V) = G(V) + \bar{G}(V),
\]
where ${G}(V)_\sigma\in C^\infty (M,S^2(T'M_\sigma))$ and $\bar{G}(V)_\sigma\in C^\infty (M,S^2(T''M_\sigma))$.

%where $\bar{G}(V)\in C^\infty (M,\bar{T}_\sigma \otimes \bar{T}_\sigma)$.

Recalling the identity $g = \omega \cdot I$, we obtain a formula for the variation of the K\"ahler metric:
\[
V[g] = \omega \cdot V[I] = \omega \cdot \tilde{G}(V)\cdot \omega,
\]
and the $(1,1)$-part of $V[g]$ vanishes because of the types of $\omega$ and $\tilde{G}(V)$.

Before defining the Hitchin connection, we need to define a certain differential operator associated to a bivector field. 
From a symmetric holomorphic bivector field $Z \in C^\infty (M,S^2(T'M_\sigma))$ we can obtain a holomorphic bundle map $Z \colon T'M_\sigma^* \to T'M_\sigma$ by contraction. We define the operator $\Delta_Z$ to be the composition:
\begin{multline*}
C^\infty(M,\calL^k) \xrightarrow{\nabla_\sigma^{(1,0)}} C^\infty(M, T'M_\sigma^* \otimes \calL^k) \xrightarrow{Z \otimes \Id} C^\infty(M, T'M_\sigma \otimes \calL^k)\\
\xrightarrow{\tilde{\nabla}_\sigma^{(1,0)} \otimes \Id + \Id \otimes \nabla_\sigma^{(1,0)}} C^\infty(M,T'M_\sigma^* \otimes T'M_\sigma \otimes \calL^k) \to C^\infty(M,\calL^k),
\end{multline*}
where $\tilde{\nabla}_\sigma^{(1,0)}$ is the holomorphic part of the Levi-Civita connection, and the last arrow is the trace. %If $f$ is a smooth function on $M$, we get a holomorphic vector field $Z\partial_\sigma f$.

This operator can be expressed in a more concise way as follows. Define the operator
\[\label{page:nabla_square}
\nabla_{X,Y}^2 = \nabla_X\nabla_Y - \nabla_{\nabla_X Y},
\]
which is tensorial and symmetric in the vector fields $X$ and $Y$. Hence it can be evaluated on a symmetric bivector field and we have:
\[
\Delta_Z = \nabla_Z^2 + \nabla_{\delta(Z)},
\]
where $\delta(Z)$ denotes the divergence of the bivector field $Z$.

The previous construction can be done for any line bundle $\calL$ over $M$. In particular, if we consider the trivial line bundle over $M$ with the trivial connection, then the sections are just functions on $M$, and the operator $\Delta_{\tilde{g}}$ (where $\tilde{g}$ denotes the bivector field obtained by raising both indices of the metric tensor) coincides with the Laplace--Beltrami operator $\Delta$.
\label{page:Delta_Z}

\subsection{Holomorphic and rigid families}

The explicit construction of a Hitchin connection in \cite{andersen2012hitchin} is for a compact symplectic manifold equipped with a smooth family of K\"ahler structures that satisfy two additional properties, which we shall explain below.

Assuming that the manifold $\T$ has a complex structure, it makes sense to require the family $I$ is a holomorphic map from $\T$ to the space of complex structures. We make this requirement precise as follows:

\begin{defn}
Let $\T$ be a complex manifold, and $I$ a family of K\"{a}hler structures on $M$ that is parametrized by $\T$. We say that $I$ is \emph{holomorphic} if:
\[
V'[I] = V[I]' \quad \text{and} \quad V''[I] = V[I]'',
\]
for any vector field $V$ on $\T$.
\end{defn}

The second condition is the \emph{rigidity} of the family of K\"{a}hler structures.

\begin{defn}\index{K\"ahler structure!family!rigid}
We say that the family $I$ of K\"{a}hler structures on $M$ is \emph{rigid} if
%\bar{\partial}_\sigma(G(V)_\sigma) = 0
\begin{equation}\label{eq:rigidity_condition}
\nabla_{X''} G(V) = 0,
\end{equation}
for all vector fields $V$ on $\T$ and $X$ on $M$.
\end{defn}

In other words, the family $I$ is rigid if $G(V)$ is a holomorphic section of $S^2(T'M)$,
for any vector field $V$ on $\T$. 

% Paolo's edit, original version is below.
\begin{remark}
The expression {\em rigid family} is used in this context for the following reason (the notion was first introduced in \cite{andersen2012hitchin}): it might be possible to extend a rigid family to a bigger family of rigid structures whose dimension at $\sigma$ in the family is given by $\dim H^0(M, S^2(T'M))$, but beyond this the family does not deform any further. Thus it is rigid in this sense. 
\end{remark}

% Original version:
%\begin{remark}
%The use the word {\em rigid family} in this context for the following reason (the notion was %first introduced in \cite{andersen2012hitchin}): It may be possibly be extended a rigid family %to a bigger family of rigid structures whose dimension at $\sigma$ in the family is given by %$\dim H^0(M, S^2(T'M)$. But beyond this the family does not deform any further. Thus it is rigid %in this sense. 
%\end{remark}

\subsection{The Hitchin connection}

The prequantum space $\calP_k = C^\infty(M, \calL^k )$ forms the fibre of a trivial vector bundle over $\T$ of infinite rank,
\begin{equation}
\hat{\calP}_k = \T \times \calP_k.
\end{equation}
Let $\nabla^t$ denote the trivial connection on this bundle.

\begin{defn}\index{Hitchin connection|textbf}
A \emph{Hitchin connection} in the bundle $\hat{\calP}_k$ is a connection of the form
\begin{equation}\label{eq:hitchin_connection_trivial_plus_a}
\nabla = \nabla^t + a,
\end{equation}
where $a \in \Omega^1(\T, \D(M, \calL^k ))$ is a one-form on $\T$ with values in the space of differential operators on sections of $\calL^k$, such that $\nabla$ preserves the quantum subspaces
\[
\calQ_k(\sigma) = H^0(M_\sigma , \calL^k )
\]
of holomorphic sections of the $k$-th power of the prequantum line bundle, inside each fibre of $\hat{\calP}_k$.
\end{defn}

The existence of a Hitchin connection in the bundle $\hat{\calP}_k$ implies that the subspaces $\calQ_k(\sigma)$ form a subbundle $\hat{\calQ}_k$, because it can be trivialized locally through parallel transport by $\nabla$.

We are now ready to state the main result of \cite{andersen2012hitchin}, showing the existence of the Hitchin connection.

\begin{thm}[Andersen]\label{thm:existence_h_connection}
Let $(M,\omega)$ be a compact, prequantizable, symplectic manifold and assume that $H^1(M;\R)=0$ as well as that there is an $n \in \Z$ such that the first Chern class of $(M, \omega)$ coincides with $n\left[\frac{\omega}{2\pi}\right] \in H^2(M;\Z)$. Moreover suppose that $I$ is a rigid holomorphic family of K\"{a}hler structures on $M$, parametrized by a complex manifold $\T$. Then there exists a Hitchin connection in the bundle $\hat{\calQ}_k$ over $\T$, given by the following expression:
\[
\hat{\nabla}_V = \nabla_V^t + \frac{1}{4k+2n} \{ \Delta_{G(V)} + 2\nabla_{G(V)\cdot dF} + 4kV'[F]\},
\]
where $\nabla_V^t$ is the trivial connection in $\hat{\calP}_k$, and $V$ is any smooth vector field on $\T$.
\end{thm}

\section{Formal connections}\label{section: formal connections}

The idea of formal connections arises as a generalization of the formal Hitchin connection that was defined in \cite{andersen2012hitchin} as the analogue of the Hitchin connection from geometric quantization. In that paper it was  shown that the formal Hitchin connection is flat under certain conditions. Its trivialization up to first order has been given  by the first named author of this paper together with Gammelgaard in \cite{andersen2011hitchin}.

Let $M$ be a symplectic manifold and $\T$ a smooth manifold parametrizing a family of star products on $M$.
We let $C_h$ be the trivial fibre bundle over $\T$ with fibre $C^\infty(M)[[h]]$, i.e.
\[
C_h = \T \times C^\infty(M)[[h]].
\]

Let $\D(M)$ denote the space of differential operators on $M$, and let $\D_h(M) = \D(M)[[h]]$ denote formal differential operators on $M$, which are formal power series with coefficients in $\D(M)$.

\begin{defn}
A \emph{formal connection} $D$ is a connection in the bundle $C_h$ over $\T$ that can be written as
\begin{equation}\label{eq:formal_connection}
D_V f = V[f] + A(V)(f),
\end{equation}
where $A$ is a smooth $1$-form on $\T$ with values in $\D_h(M)$ such that $A = 0 \, \mathrm{mod} \, h$, $f$ is a smooth section of $C_h$, $V$ is any smooth vector field on $\T$, and $V[f]$ denotes the derivative of $f$ along $V$.
\end{defn}

The operator $A(V)(f)$ can be expressed as a series of differential operators
\[
A(V)(f) = \sum_{k=1}^\infty A^{k}(V)(f)h^k,
\]
where each $A^{k}$ is a smooth $1$-form on $\T$ with values in $\D(M)$.

Normally we are interested in looking at formal connections in the presence of a family of star products on the manifold, and then we require the following compatibility:

\begin{defn}\label{defn:formal_connections_for_family_of_products}
Let $\{\star_\sigma\}_{\sigma\in \T}$ be a family of star products on $M$. We say that a formal connection $D$ is \emph{compatible with the family of star products} $\{\star_\sigma\}_{\sigma \in \T}$ if $D_V$ is a derivation of $\star_\sigma$ for every vector field $V$ and every $\sigma \in \T$, that is, the following equality holds:
\begin{equation}\label{eq:formal_connection_is_derivation}
D_V(f \star_\sigma g) = D_V(f) \star_\sigma g + f \star_\sigma D_V( g) 
\end{equation}
for all smooth sections $f$ and $g$ of $C_h$.
\end{defn}

If the family of star products is natural, we also require the $1$-form $A$ to consist of natural differential operators,
i.e. the degree as a differential operator of the component $A^k$ is bounded by $k$.

\subsection{The formal Hitchin connection associated to geometric quantization}\label{subsection: formal Hitchin}

\begin{defn}\label{defn:formal_hitchin_connection}
Let $M$ be a symplectic manifold with a family of compatible almost complex structures parametrized by a complex manifold $\T$, so that for any $\sigma \in \T$, the manifold $M_\sigma$ is K\"ahler. Let $\{\star_\sigma^{BT}\}_{\sigma\in \T}$ be the associated family of Berezin-Toeplitz star products, see Subsection \ref{subsection: BT}.
 A \emph{formal Hitchin connection} for $\T$ is a formal connection which is compatible with  $\{\star_\sigma^{BT}\}_{\sigma\in \T}$ and which is flat.
\end{defn}

The Hitchin connection $\hat{\nabla}$ in $\hat{\calQ}_k$ which we discussed in the previous section induces a connection $\hat{\nabla}^e$ in the endomorphism bundle $\End(\hat{\calQ}_k)$. The following result establishes the existence of a \emph{formal Hitchin connection} under the same assumptions as in Theorem \ref{thm:existence_h_connection}.

\begin{thm}[Andersen]\label{thm:formal_h_connection}\index{formal connection!Hitchin|textbf}
There is a unique formal connection $D$, written as $D_V = V + \tilde{A}(V)$, which satisfies
\begin{equation}
\hat{\nabla}_V^e T_f^{(k)} \sim T_{(D_V f)(1/(2k+n))}^{(k)}
\end{equation}
for all smooth sections $f$ of $C_h$ and all smooth vector fields $V$ on $\T$. 
Here the symbol $\sim$ has the following meaning: for any positive integer $L$ we have that
\[
\Bigg\lVert \hat{\nabla}_V^e T_f^{(k)} - \left( T_{V[f]}^{(k)} + \sum_{l=1}^L T_{\tilde{A}_V^{(l)}f}^{(k)}\frac{1}{(2k+n)^l} \right) \Bigg\rVert = O(k^{-(L+1)})
\]
uniformly over compact subsets of $\T$ for all smooth maps $f \colon \T \to C^\infty(M)$.
\end{thm}

In \cite{andersen2012hitchin} the following explicit formula for $\tilde{A}$ is given:
\begin{equation}\label{expression: HC}
\tilde{A}(V)(f) = -V[F]f + V[F] \star^{BT} f + h(E(V)(f) - H(V)\star^{BT} f),
\end{equation}
where $E$ is a $1$-form on $\T$ with values in $\D(M)$ and $H$ is a $1$-form with values in $C^\infty(M)$ given by $H(V) = E(V)(1)$. This result has been further refined by the first named author of this paper and Gammelgaard, who obtained an explicit formula for $E$ in \cite{andersen2011hitchin} which reads:
\begin{equation}\label{eq:operator_E}
E(V)(f) = -\frac{1}{4}(\Delta_{\tilde{G}(V)}(f) - 2\nabla_{\tilde{G}(V)dF}(f) -2\Delta_{\tilde{G}(V)}(F)f -2nV[F]f).
\end{equation}
From this equation we immediately get an expression for $H$:
\[
H(V) = E(V)(1) = \frac{1}{2}(\Delta_{\tilde{G}(V)}(F) + nV[F]).
\]

We can summarize the previous results by writing the following formula for the formal Hitchin connection studied by Andersen:
\begin{equation}\label{eq:formal_h_connection_explicit}
\begin{split}
D_V f =& V[f] -\frac{1}{4}h\Delta_{\tilde{G}(V)}(f) + \frac{1}{2}h\nabla_{\tilde{G}(V)dF}(f) + V[F]\star^{BT}f -V[F]f \\
& -\frac{1}{2}h(\Delta_{\tilde{G}(V)}(F)\star^{BT} f -nV[F]\star^{BT}f -\Delta_{\tilde{G}(V)}(F)f -nV[F]f).
\end{split}
\end{equation}

The following two propositions, proved in \cite{andersen2012hitchin}, assert that the formal connection constructed in Theorem \ref{thm:existence_h_connection} is a derivation with respect to the Berezin-Toeplitz star product (thus it is compatible with the family of Berezin-Toeplitz star products) and that it is flat whenever the Hitchin connection is projectively flat.

\begin{prop}\label{prop:FHC_is_derivation}\index{formal connection!derivation property}
The formal operator $D_V$ is a derivation with respect to the star product $\star_\sigma^{BT}$ for each $\sigma \in \T$, meaning that it satisfies the relation:
\begin{equation}\label{eq:derivation_relation}
D_V(f_1 \star^{BT} f_2) = D_V(f_1 ) \star^{BT} f_2 + f_1 \star^{BT} D_V(f_2)
\end{equation}
for all $f_1,f_2 \in C^\infty(M)$. 
\end{prop}

\begin{prop}\label{prop:hitchin_proj_flat_formal_flat}
If the Hitchin connection $\hat{\nabla}$ in $\hat{\calQ}_k$ is projectively flat, then the formal Hitchin connection 
$D_V = V + \tilde{A}(V)$
associated to it is flat.
\end{prop}

\begin{remark}
Proposition \ref{prop:FHC_is_derivation} relies on the theory of geometric quantization and Toeplitz operators, and their link to deformation quantization. Consequently its validity can be traced back to the existence of a Hitchin connection in geometric quantization, which puts many requirements on the objects involved -- in particular,
one needs the compatible K\"ahler structures to be rigid and holomorphic.
That is, if we adopt the assumptions of Theorem \ref{thm:existence_h_connection} and if the Hitchin connection $\hat{\nabla}$ is projectively flat, the previous two propositions imply that formal Hitchin connection from \cite{andersen2012hitchin} is a formal Hitchin connection according to Definition \ref{defn:formal_hitchin_connection}.
\end{remark}

\subsection{Low orders of the formal Hitchin connection}

The explicit expression of the formal Hitchin connection \eqref{eq:formal_h_connection_explicit} makes sense in a more general setting. Therefore the question arises to which extent this expression actually defines a formal connection 
compatible with the Berezin-Toeplitz star products.
In other words, we wonder whether it defines a derivation of the Berezin-Toeplitz
star products for a general family of K\"ahler structures.

Here we shall answer this question up to order one by a direct computation, for which we need not assume that a Hitchin connection in the framework of geometric quantization exists.
We include two preliminary lemmata concerning the coefficients of the Berezin-Toeplitz star product.

Recall that we use the special notation $c^{(k)}$ for the coefficients of $\star_\sigma$, when we consider a Berezin-Toeplitz star product.

\begin{lemma}\label{lemma:formula_for_V[c1]_proposition}
Let $M$ be a symplectic manifold with a family of compatible K\"ahler structures parametrized on $M$, by a manifold $\T$, and let $c^{(k)}$ denote the coefficients of the Berezin-Toeplitz star product associated to the complex structure for a certain $\sigma \in \T$. Then we have:
\begin{equation}\label{eq:formula_for_V[c1]_proposition}
V[c^{(1)}](f,g) = \frac{1}{4}\left( \Delta_{\tilde{G}(V)}(fg) -\Delta_{\tilde{G}(V)}(f)g -\Delta_{\tilde{G}(V)}(g)f \right).
\end{equation}
\end{lemma}
\begin{proof}
By a result of Karabegov \cite{karabegov1996deformation} we know that the degree $1$ coefficient of the Berezin-Toeplitz star product can be written as:\index{star product!Berezin-Toeplitz}
\begin{equation}\label{eq:formulae_for_c1_chap7}
c^{(1)}(f,g) =  g({\partial}f,\bar{\partial} g) = i\nabla_{X''_{g}} (f),
\end{equation}
for any functions $f,g \in C^\infty(M)$, where $X_{f}$ is the Hamiltonian vector field associated to $f$.

By differentiating equation \eqref{eq:formulae_for_c1_chap7}, we get the following relation:
\begin{equation}\label{eq:formulae_for_V[c1]}
V[c^{(1)}](f,g) = \frac{1}{2}df \tilde{G}(V) dg = \frac{1}{2} \nabla_{\tilde{G}(V)dg} (f).
\end{equation}
The operator $\Delta_{\tilde{G}(V)}$ is written as $\Delta_{\tilde{G}(V)} = \nabla_{\tilde{G}(V)}^2 + \nabla_{\delta {\tilde{G}(V)}}$, thus we get that
\[
 \Delta_{\tilde{G}(V)}(fg) -\Delta_{\tilde{G}(V)}(f)g -\Delta_{\tilde{G}(V)}(g)f  = \nabla^2_{\tilde{G}(V)}(fg) -\nabla^2_{\tilde{G}(V)}(f)g -\nabla^2_{\tilde{G}(V)}(g)f,
\]
since the sum of the order one terms vanishes. We can express the symmetric bivector field $\tilde{G}(V)$ as $\sum_j (X_j \otimes Y_j)$ for vector fields $X_j$ and $Y_j$, and rewrite the right hand side of \eqref{eq:formula_for_V[c1]_proposition} as:
\begin{align*}
&\frac{1}{4}\left( \sum_j \nabla_{X_j}\nabla_{Y_j}(fg) -g\sum_j \nabla_{X_j}(\nabla_{Y_j}f) -f \sum_j \nabla_{X_j}(\nabla_{Y_j}g) \right) \\
&= \frac{1}{4} \sum_j \left( \nabla_{Y_j}(f)\nabla_{X_j}(g) + \nabla_{X_j}(f)\nabla_{Y_j}(g) \right)
=  \frac{1}{2} \sum_j \left( \nabla_{X_j}(f)\nabla_{Y_j}(g) \right) \\
&= \frac{1}{2}df \tilde{G}(V) dg,
\end{align*}
where we use the symmetry of the bivector field. This concludes the proof.
\end{proof}

\begin{remark}
Let us note that the expression we obtained for $V[c^{(1)}]$ also shows that it is symmetric in the two variables.
\end{remark}

We remark that because the Berezin-Toeplitz star product is natural, the first oder coefficient
$c^{(1)}$ is a differential operator of order $1$ and hence is a derivation with respect to both arguments.

We now show that the expression \eqref{eq:formal_h_connection_explicit} gives a derivation of the Berezin-Toeplitz star product up to order $1$ in $h$. The derivation relation \eqref{eq:derivation_relation} can be written as:
\begin{equation}\label{eq:derivation_relation_revisited}
f\ V[\star^{BT}]\ g = \tilde{A}(V)(f)\star^{BT} g + f\star^{BT} \tilde{A}(V)(g) - \tilde{A}(V)(f \star^{BT} g),
\end{equation}
where $V[\star^{BT}]$ denotes the product with coefficients $V[c^{(k)}]$.  Modulo $h^2$ and using $\tilde{A}^0 = 0$, we arrive at the following condition that has to hold for all vector fields $V$ on $\T$ and smooth functions $f$ and $g$ on $M$:
\begin{equation}\label{eq:derivation_relation_at_order_one}
V[c^{(1)}](f,g) = -\tilde{A}^1(V)(fg) + \tilde{A}^1(V)(f)g + f\tilde{A}^1(V)(g).
\end{equation}
To check this, we extract the expression for $\tilde{A}^1$ from \eqref{eq:formal_h_connection_explicit}, which gives us the following:
\[
\tilde{A}^1(V)(f) = -\frac{1}{4}\Delta_{\tilde{G}(V)}(f) + c^{(1)}(V[F],f) +V[c^{(1)}](F,f).
\]
We can now substitute this into the right-hand side of equation \eqref{eq:derivation_relation_at_order_one} and get:
\[
 \frac{1}{4}\left( \Delta_{\tilde{G}(V)}(fg) - g\Delta_{\tilde{G}(V)}(f) - f\Delta_{\tilde{G}(V)}(g)\right).
\]
Observe that the second and third terms in $\tilde{A}^1(V)$ do not appear, since they are differential operators
of order $1$ and hence automatically satisfy Leibniz rule. We finally observe that what we obtained is precisely the expression from Lemma \ref{lemma:formula_for_V[c1]_proposition}, and so we have checked that the formal connection corresponding to $\tilde{A}^1$ is compatible
with $\star^{BT}$ modulo terms of order $h^2$.

Moreover we can check that the expression \eqref{eq:formal_h_connection_explicit} defines a formal connection that is flat up to order one in $h$. This amounts to showing that its curvature vanishes modulo $h^2$. Note that the only terms we see when we compute the curvature modulo $h^2$ are those coming from $d_\T\tilde{A}^1$, since the commutator terms in the curvature are of order at least $h^2$. Therefore $\tilde{A}^1$ defines a flat connection up to order one if an only if it is closed with respect to $d_\T$. The closeness is most easily established by defining the $0$-form $P_1 = \frac{1}{4}\Delta -c^{(1)}(F,f)$ and checking that $V[-P_1] = \tilde{A}^1(V)$ for any vector field on $\T$.

The discussion above can be summed up in the following proposition.

\begin{prop}\label{prop: order 1 Hitchin}
Let $M$ be a symplectic manifold with a family of compatible K\"ahler structures parametrized by a complex manifold $\T$. Then expression \eqref{eq:formal_h_connection_explicit} defines a formal connection that, up to order one in the formal parameter, is a derivation of the family of Berezin-Toeplitz star products on $M$ and flat . Therefore it defines a formal Hitchin connection in the sense of Definition \ref{defn:formal_hitchin_connection} modulo terms of order $h^2$.
\end{prop}

\subsection{Derivations of star products}

We shall now put aside the formal Hitchin connection and look at formal connections in general. We aim at describing the space of formal connections on a symplectic manifold.

We begin by studying the space of derivations of a star product. Recall that a map $B \colon C_h^\infty(M) \to C_h^\infty(M)$ is a derivation with respect to a fixed star products if it satisfies the relation:\index{star product!derivation}
\begin{equation}\label{eq:derivation_relation_defined}
B(f \star g) = B(f) \star g + f \star B(g),
\end{equation}
for any $f, g$ smooth (formal) functions on $M$.

Gutt and Rawnsley showed the following proposition in \cite{gutt1999equivalence}:
\begin{prop}\label{prop:GuttRawnsley}
%Let $\star$ be a star product on a symplectic manifold $(M,\omega)$. The derivations of $\star$ are in bijection with formal symplectic vector fields on $M$. If $B = \sum_{k\in\N} B_k h^k$ is a derivation of $\star$, then $B_k$ corresponds to a symplectic vector field $Y_k$ for each $k$. The correspondence can be made explicit on any contractible open subset $U$ of $M$: if locally $Y(f)|_U = \{b,f\}
On a symplectic manifold $(M,\omega)$ with a star product $\star$, any derivation $B$ of $\star$ can be written as  $B = \sum_{k\in\N} B_k h^k$ where each $B_k$ corresponds to a symplectic vector field $Y_k$ on $M$. The correspondence can be made explicit on any contractible open subset $U \subset M$ by the following relation:
\[
B_k(f)|_U = \frac{1}{h} (b_k \star f - f \star b_k),
\]
where $b_k \in C^\infty(U)$ is such that $Y_k(f)|_U = \{b_k,f\}|_U$. 
\end{prop}

We say that a formal symplectic vector field is a formal series $Y = \sum_{k\in\N} Y_k h^k$, where $Y_k$ is a symplectic vector field for all $k$. Then the proposition above is telling us that derivations correspond to formal symplectic vector fields and that locally a derivation can be written as $\frac{1}{h} \ad_\star b$ for a formal smooth function $b$.

In general we have an isomorphism:
\[
\Gamma_\mathrm{sym}(M)/\Gamma_\mathrm{Ham}(M) \cong H^1(M;\R),
\]
where $\Gamma_\mathrm{sym}(M)$ and $\Gamma_\mathrm{Ham}(M)$ denote the space of symplectic and Hamiltonian vector fields on $M$, respectively. 
Therefore, if $H^1(M;\R) = 0$, all symplectic vector fields are Hamiltonian, and hence, every derivation corresponds to a formal Hamiltonian vector field $X_b = \sum_{k\in\N} X_{b_k} h^k$, for a formal smooth function $b$ in that case. More explicitly, any derivation can be written globally in the form $\frac{1}{h}\ad_\star b$ for a formal function $b$. 
Observe that ther kernel of $b \mapsto \frac{1}{h}\ad_\star b$ are exactly the constant formal functions.

\subsection{The affine space of formal connections}

Let $D$ and $D'$ be two formal connections on $M$ for the same family of star products parametrized by $\T$.
It is immediate to see that
\[
D'_V - D_V = A'(V) - A(V) = (A'-A)(V).
\]
Hence their difference is a $1$-form on $\T$ with values in the derivations of the star products of the family, and it is zero modulo $h$.
In the following $\Der(M,\star)$ denotes the space of derivations of the star product $\star$ on $M$ and
$\Der_0(M,\star)$ denotes the subset of derivations that are trivial modulo $h$.

Given a symplectic manifold $M$ equipped with a family of star products $\{\star_\sigma\}_{\sigma\in\T}$ parametrized by $\T$, we denote by $\F(M,\star_\sigma)$ the space of the formal connections that are compatible with the family.
We see that $\F(M,\star_\sigma)$ is an affine space over the space of $1$-forms on $\T$ with values in $\Der_0(M,\star)$, and thus can be written as:
\begin{equation}\label{eq:affine_space_formal_connections_1}
\F(M,\star_\sigma) = D_0 + \Omega^1(\T,\Der_0(M,\star_\sigma)),
\end{equation}
for a fixed formal connection $D_0$, which is compatible with $\{\star_\sigma\}_{\sigma \in \T}$.

As remarked above, the derivations of $\star$ correspond to formal symplectc vector fields on $M$, therefore we can rewrite \eqref{eq:affine_space_formal_connections_1} in the following way:
\[
\F(M,\star_\sigma) = D_0 + \Omega^1(\T,h\Gamma_\mathrm{sym}(M)[[h]]).
\]

If we assume that $H^1(M;\R)$ vanishes, all derivations of $\star$ are essentially inner, and therefore they are parametrized by an element in $\tilde{C}_h^\infty(M)$, the space of formal functions on $M$ modulo the constants. Therefore the compatible formal connections form an affine space modelled on the $1$-forms on $\T$ with values in $h\tilde{C}_h^\infty(M)$.
\[
\F(M,\star_\sigma) \cong D_0 + \Omega^1(\T,h\tilde{C}_h^\infty(M)), 
\]
for a fixed compatible formal connection $D_0$.

\subsection{Gauge transformations of formal connections}

We shall study gauge transformations on the space of formal connections $\F(M,\star_\sigma)$. The transformations we consider are self-equivalences of the family of star products, since the connections should still act as derivations after we transform them. This means that we look at $P \in C^\infty(\T,\D_h(M))$ with $P=\mathrm{id} \textrm{ mod } h$ such that
\begin{equation}\label{eq:self_equivalence_star_product}
P_\sigma(f \star_\sigma g) = P_\sigma(f) \star_\sigma P_\sigma(g),
\end{equation}
for any $\sigma \in \T$ and any smooth function $f$ and $g$.

We are now ready to prove the following theorem.

\begin{thm}\label{thm:gauge_transformation_of_flat_formal_connections}
Let $M$ be a symplectic manifold with a family of star products $\{\star_\sigma\}_{\sigma\in\T}$ parametrized by a smooth manifold $\T$ with $H^1(\T,\mathbb{R})=0$. Let $D,D' \in \F(M,\star_\sigma)$ be formal connections for the family and let us assume that they are flat. Then they are gauge equivalent via a self-equivalence of the family of star products $P \in C^\infty(\T,\D_h(M))$, meaning that 
\begin{equation}\label{eq:gauge_equivalence_formal_connections}
D_V' = P^{-1}D_VP,
\end{equation}
for any vector field $V$ on $\T$.
\end{thm}

\begin{proof}
As usual we can write the formal connections in the form:
\begin{align*}
D_V &= V + A(V) \\
D_V &= V + A'(V),
\end{align*}
for two $1$-forms $A, A' \in \Omega(\T,\D_h(M))$ with values in formal differential operators on $M$, and any vector field $V$ on $\T$. Then we can rewrite \eqref{eq:gauge_equivalence_formal_connections} by plugging in a section $f$ of the bundle as:
\begin{equation}\label{eq:gauge_equivalence_formal_connections_2}
\begin{split}
V[f] + A'(V)(f) &= P^{-1}(V+A(V))P(f) \\
&= P^{-1}\left( V[P](f) + P(V[f]) + A(V)(P(f)) \right).
\end{split}
\end{equation}
Therefore if we apply $P$ on both sides we get the following equation: 
\begin{equation}\label{eq:gauge_equivalence_formal_connections_3}
V[P] = PA'(V) - A(V)P.
\end{equation}
If we can find a $P = \sum_{k\in\N} P_k h^k$ that solves the equation, then we get the wanted gauge transformation. To do so we proceed inductively. By definition we have $P_0=\mathrm{id}$ and hence $V[P_0] = 0$.
Let us assume that we have determined $P^{(l)} = \sum_{k\leq l} P_k h^k$ such that
\[
V[P^{(l)}] = P^{(l)}A'(V) - A(V)P^{(l)} + O(h^{l+1}).
\]
This can be written as:
\begin{equation}\label{eq:gauge_equivalence_formal_connections_3.1}
B_{l+1}(V)h^{l+1} = V[P^{(l)}] - (P^{(l)}A'(V) - A(V)P^{(l)}) + O(h^{l+2}),
\end{equation}
where $B_{l+1}$ is a $1$-form on $\T$ with values in differential operators on $M$.
Let us define a $1$-form $\alpha_{l}$ on $\T$ with values in formal differential operators on $M$ by $\alpha_{l}(V) = (P^{(l)}A'(V) - A(V)P^{(l)})$. We want to show that $\alpha_l$ is closed modulo $h^{l+2}$. Let  $V$ and $W$ be two commuting vector fields on $\T$. Then we have that:
\begin{equation}\label{eq:gauge_equivalence_formal_connections_4}
\begin{split}
& d_\T\alpha_l(V,W)  \\
&= V[\alpha(W)] - W[\alpha(V)] \\
&= V[P^{(l)}A'(W) - A(W)P^{(l)}] -W[P^{(l)}A'(V) - A(V)P^{(l)}] \\
&= P^{(l)}\left( A'(V)A'(W) - A'(W)A'(V) + V[A'(W)] -W[A'(V)] \right) \\
&\quad - \left( A(V)A(W) - A(W)A(V) + V[A(W)] -W[A(V)] \right) P^{(l)} \\
&\quad +h^{l+1}(B_{l+1}(V)A'(W) - A(W)B_{l+1}(V) - B_{l+1}(V)A'(V) + A(V)B_{l+1}(V)),
\end{split}
\end{equation}
where we substituted the expression for $V[P^{(l)}]$ and $W[P^{(l)}]$ again in order to obtain the last equality. Note also that the following expression, which appears in the last line of the equation,
\[
B_{l+1}(V)A'(W) - A(W)B_{l+1}(V) - B_{l+1}(V)A'(V) + A(V)B_{l+1}(V)
\]
is of oder $h$. Let us now compute the expressions for the curvature of $D$, which we are assuming is flat:
\[
0 = F_D(V,W) = D_VD_W - D_WD_V - D_{[V,W]}.
\] 
The last summand vanishes because we chose commuting vector fields, hence we get, for any section:
\begin{equation}\label{eq:gauge_equivalence_formal_connections_5}
\begin{split}
0 
&= A(V)A(W)(f) - A(W)A(V)(f)  - W[A(V)(f)] \\
&\quad+ A(V)W[f] + V[A(W)(f)] - A(W)V[f] \\
&= A(V)A(W)(f) - A(W)A(V)(f) -W[A(V)](f) + V[A(W)](f),
\end{split}
\end{equation}
which is the same as:
\[
0 = F_D(V,W) =  A(V)A(W) - A(W)A(V) + V[A(W)] -W[A(V)] .
\]
By computing the curvature in the same way for $D'$ we obtain:
\[
0 = F_{D'}(V,W) = A'(V)A'(W) - A'(W)A'(V) + V[A'(W)] -W[A'(V)] .
\]
By comparing with \eqref{eq:gauge_equivalence_formal_connections_4}, we see that $d_\T \alpha_l = 0 \pmod{h^{l+2}}$, and therefore, by \eqref{eq:gauge_equivalence_formal_connections_3.1}, we also have that $d_\T B_{l+1} = 0$. Since $H^1(\T;\R)$ is trivial, we can find a smooth function $P_{l+1} \colon \T \to \calD(M)$ such that $V[P_{l+1}h^{l+1}] = -B_{l+1}h^{l+1}$.\footnote{
The existence of $P_{l+1}$ follows from the existence of
an operator
\[
d^*: d(\Omega^{k-1}(\T)) \to \Omega^{k-1}(\T),
\]
called the anti-differential,
such that $d d^*\beta = \beta$ holds,
and which maps smooth families to smooth families.
For $\T$ compact, $d^*$ might be constructed using Hodge-theory.
For arbitrary $\T$ one can use the \v{C}ech-de Rham double complex to construct such an operator, as is done in \cite{gustavo2014}.} 
We now set $P^{(l+1)} = P^{(l)} + P_{l+1}h^{l+1}$ as the notation suggests. To conclude the inductive step and the proof it is enough to show that:
\[
V[P^{(l+1)}] - (P^{(l+1)}A'(V) - A(V)P^{(l+1)}) = 0 \pmod{h^{l+2}}.
\]
By expanding the left-hand side we get:
\[
V[P^{(l)}] - (P^{(l)}A'(V) - A(V)P^{(l)}) - h^{l+1}B_{l+1} 
+ h^{l+1} (P_{l+1}A'(V) - A(V)P_{l+1}),
\]
which is a multiple of $h^{l+2}$ since $P_{l+1}A'(V) - A(V)P_{l+1}$ is of order $h$.
\end{proof}

\begin{remark}
At first glance Theorem \ref{thm:gauge_transformation_of_flat_formal_connections}
seems too strong, since one might naively expect the vanishing of the fundamental group,
instead of the first cohomology, to be the relevant condition.
However two factors improve the situation: First of all, the bundle under consideration $C_h$ is assumed to be trivial.
Second, the Lie algebra in which our connection takes values in is the Lie algebra of formal differential operators
which vanish modulo $h$. This Lie algebra is filtered by those formal differential operators which start
at order $h^k$. The associated graded Lie algebra is easily seen to be abelian.
This two facts show that we are effectively dealing with the case of connection with values in
an abelian Lie algebra. Hence flatness reduces to closedness and gauge-equivalence
reduces to a shift by an exact one-form. This explains the purely cohomological nature of our result.
\end{remark}

%As seen earlier in Proposition \ref{prop:hitchin_proj_flat_formal_flat}, Andersen showed \cite{andersen2012hitchin} that the %formal Hitchin connection associated to the Hitchin connection of geometric quantization is flat if the Hitchin connection is %projectively flat. 

We record the following consequence of our considerations:

\begin{corollary}
Let $\T$ be a manifold with trivial first cohomology group, i.e. $H^1(\T,\mathbb{R})=0$. If there exists a formal Hitchin connection $D$ in the bundle $C_h = \T \times C_h^\infty(M)$ on $\T$, then it is unique up to gauge equivalence.
\end{corollary}

\section{Existence of formal connections}\label{section: existence}

In this section we study the question of existence of formal connections on a symplectic manifold $(M,\omega)$,  equipped with a smooth family of natural star products $\{\star_\sigma\}_{\sigma \in \T}$. As we will see, this problem can be reduced to a cohomological condition in terms of the corresponding family $\cl(\star_\sigma)$
of characteristic $2$-forms, see Subsection \ref{subsection: Fedosov}.

We briefly review the Hochschild complex $HC^\bullet(A,A)$
of an algebra $A$ with values in itself, relying on \cite{Gerstenhaber,waldmann2007poisson}.
It is given by
$$ HC^k(A,A) := \mathrm{Hom}(A^{\otimes k},A),$$
the space of multilinear maps from $A$ to itself.
The graded vector space $HC^\bullet(A,A)$ comes equipped
with the Gerstenhaber bracket $[\cdot,\cdot]_{G}$.
For $\psi \in HC^{r+1}(A,A)$ and $\phi \in HC^{s+1}(A,A)$, it is given by
\begin{eqnarray*}
 [\psi, \phi](a_0,\dots, a_{r+s})
:= \sum_{i=0}^r(-1)^{is}\psi(a_0,\dots,a_{i-1},\phi(a_i,\dots,a_{i+s}),
a_{i+s+1},\dots, a_{r+s}) \\
-(-1)^{rs} \sum_{j=0}^s (-1)^{jr}\phi(a_0,\dots,a_{j-1},\psi(a_j,\dots,a_{j+r}),
a_{j+r+1},\dots, a_{r+s}).
\end{eqnarray*}
The bracket $[\cdot,\cdot]_{G}$ makes $HC(A,A)$ into a graded Lie algebra
if we assign to an element $\psi\in HC^{r+1}(A,A)$ the degree $r$.
Moreover, the associativity of the product $m: A\otimes A \to A$ of $A$
can be rewritten as
$$ [m,m]_G=0.$$
As a consequence, the map $d_H(\phi):= [m,\phi]_G$ defines a coboundary operator on $HC^\bullet(A,A)$,
which is referred to as the Hochschild differential.
The cohomology of $HC^\bullet(A,A)$ with respect to $d_H$ is called
the Hochschild cohomology of $A$ and is known to control the infinitesimal
deformations of the algebra $A$.

We will be interested in the case of $A$ being $C^\infty_h(M)$,
seen as a module over $\mathbb{C}[[h]]$.
To get a reasonable cohomology theory,
one has to take the Fr\'echet-topology into account (thus modifying the tensor
product and the space of homomorphisms). Alternatively, one can restrict
to the space of multi-differential operators $\D^m(M)$, where $m$ is the number of arguments.
 The Gerstenhaber bracket $[\cdot,\cdot]_G$ and the Hochschild differential $d_H$ restrict to $\D^\bullet(M)=\oplus_{m\ge 0}\D^m(M)$.
We will also consider the space of formal multi-differential operators $\D^\bullet_h(M)$ and the $\mathbb{C}[[h]]$-linear
extensions of $[\cdot,\cdot]_G$ and $d_H$ to $\D^\bullet_h(M)$, respectively.

Let $\{\star_\sigma\}_{\sigma \in \T}$ be a family of star products for $(M,\omega)$,
smoothly parametrized by $\T$.
Given a vector field $V$ on $\T$, 
we write $V[\star_\sigma]$ for the product whose $i$-th coefficient is $V[c^i]$,
i.e. the variation of $\{\star_\sigma\}_{\sigma \in \T}$ in the direction of $V$.
The operator
$$B_V(f,g) = f V[\star_\sigma] g$$ 
can be seen as a family of elements in $\D^2_h(M)$.
Since associativity of $\{\star_\sigma\}_{\sigma \in \T}$ can be written as
$$ [\star_{\sigma},\star_\sigma]_G = 0,$$
we obtain 
$$ d_H B_V = [\star_\sigma,V[\star_{\sigma}]]_{G} = \frac{1}{2} V([\star_{\sigma},\star_\sigma]_G) = 0,$$
i.e. $B_V$ is closed with respect to the Hochschild differential associated to the family
of star products $\{\star_\sigma\}_{\sigma \in \T}$.

We now reconsider the compatibility requirement between a formal connection $D$ and $\{\star_\sigma\}_{\sigma\in \T}$ from Definition \ref{defn:formal_connections_for_family_of_products}.
Using the decomposition $D_V = V + A(V)$,
the compatibility can be written as
\begin{equation}\label{eq:condition_on_A_for_existence_of_connection}
 A(V)(f)\star_\sigma g + f \star_\sigma A(V)(g) - A(V)(f\star_\sigma g) = f V[\star_\sigma] g.
\end{equation}
We see that the above equation can be further rewritten as
$$ d_H A(V) = V[\star_\sigma]$$
and we arrive at the following interpretation of \eqref{eq:condition_on_A_for_existence_of_connection}:

\begin{prop}\label{prop:existence_connection_equivalent_to_cochain_exact}
Let $(M,\omega)$ be a symplectic manifold, equipped with a family of natural star products $\{\star_\sigma\}_{\sigma \in \T}$ on $M$, smoothly parametrized by $\T$. 

There is a one-to-one correspondence between:
\begin{enumerate}
\item Formal connections $D_V= V + A$ on $\T$ which are compatible with $\{\star_\sigma\}_{\sigma \in \T}$.
\item Families of formal differential operators $A(V) \in \D^1_h(M)$, with
$A(V)= 0 \, \mathrm{mod} \, h$ that satisfy
\[
d_H A(V) = V[\star_\sigma],
\]
where $d_H$ denotes the Hochschild coboundary operator with respect to the family of star products $\{\star_\sigma\}_{\sigma \in \T}$.
\end{enumerate}

In particular, a family of star products $\{\star_\sigma\}_{\sigma \in \T}$
admits a compatible formal connection if and only if the family of cocycles
given by $B(\sigma,V) = V[\star_\sigma]$ is exact.
\end{prop}

By work of Weinstein-Xu \cite{WeinsteinXu}, the Hochschild cohomology
of $C_h^\infty(M)$ with respect to 
a star product is isomorphic to the de Rham cohomology $hH^\bullet(M, {\mathbb R})[[h]]$.
Consequently, one expects that the cohomological condition
on $V[\star_\sigma]$ from Proposition \ref{prop:existence_connection_equivalent_to_cochain_exact} translates into a cohomological condition
on the characteristic $2$-forms $\cl(\star_\sigma)$.
We will show that this is indeed the case. Instead of applying the results of \cite{WeinsteinXu},
we directly work in Fedosov's framework, which we modify for our purposes. 
More precisely, we aim at showing that every choice of trivialization
for the family of characteristic $2$-forms $\cl(\star_\sigma)$ leads to a compatible
formal connection of the corresponding star products $\star_{\sigma}$. We introduce
$$ {\mathcal P} := \{ \beta \in \Omega^1(\T,\Omega^1(M)[[h]]) \mid d_M i_V\beta = V[\cl(\star_\sigma)]\},$$
which we regard as the space of trivializations of all variations of $\cl(\star_\sigma)$
and
$$ {\mathcal C} := \{ A \in h\Omega^1(\T,\D^1_h(M))  \mid d_H i_V A = V[\star_\sigma]\},$$
the space of trivializations of all variations of the family $\star_\sigma$ in $\D_h^2(M)$.

\begin{thm}\label{thm:trivializations->trivializations}
Let $(M,\omega)$ be a symplectic manifold which is equipped with a family
of natural star products $\{\star_\sigma\}$, smoothly parametrized by $\T$.
There is a natural map
$$ C : {\mathcal P}  \rightarrow {\mathcal C}
$$
\end{thm}

The proof of Theorem \ref{thm:trivializations->trivializations} is postponed to the next subsection.
If we combine Theorem \ref{thm:trivializations->trivializations} and Proposition \ref{prop:existence_connection_equivalent_to_cochain_exact}, we obtain

\begin{thm}\label{thm:existence_formal_connection_iff_char_class_constant}
Let $(M,\omega)$ be a symplectic manifold which is equipped with a family
of natural star products $\{\star_\sigma\}$, smoothly parametrized by $\T$.

Then the following statements are equivalent:
\begin{enumerate}
 \item The cohomology class of the family of characteristic $2$-forms $\cl(\star_\sigma)$ is locally constant in $\T$.
 \item There is a $1$-form $A\in \Omega^1(\T,h \D^1_h(M))$ with values in formal differential operators on $M$ such that for any vector field $V$ on $\T$, and any smooth functions $f$ and $g$ on $M$ the identity
 \[
f V[\star] g =  A(V)(f)\star g + f \star A(V)(g) - A(V)(f\star g)
 \]
 holds. 
 \item The family of star products admits a formal connection.
\end{enumerate}
\end{thm}

\begin{proof}
We start with the implication $(1) \Rightarrow (2)$.
Assuming Theorem \ref{thm:trivializations->trivializations}, the only part of this statement that needs additional arguing is that,
given family of closed $2$-forms $\alpha_\sigma := \cl(\star_\sigma)$ with locally constant cohomology class,
one can find a smooth $1$-form $\beta$ on $\T$ with values in
$\Omega^1(M)$ such that for all vector fields $V$ on $\T$ the identity
$$
d_M i_V \beta = V[\alpha_\sigma]
$$
holds.
To this end, we assume without loss of generality that $\T$ is connected.
We fix a base-point $t_0 \in \T$ and consider the family
$$ \alpha - \alpha_{t_0},$$
which is exact. 
Hence there is a family of $1$-forms $\gamma$ on $M$ such that
$$ d_M \gamma = \alpha - \alpha_{t_0}$$
holds.
If we apply the Lie derivative with respect to a vector field $V$ on $\T$, we obtain:
\[
- d_M i_V(d_\T \gamma) = V[\alpha].
\]
Hence $\beta:= -d_\T \gamma$ is the desired $1$-form on $\T$ with values
in $\Omega^1(M)$.

The implication $(2) \Rightarrow (3)$ is precisely the content of Proposition \ref{prop:existence_connection_equivalent_to_cochain_exact}.

The implication $(3) \Rightarrow (1)$ relies on the parallel transport of
the connection $D_V = V + A(V)$.
Given two points $\sigma$ and $\sigma'$ in the same connected component of
$\T$, we can choose a smooth path $\gamma$ which starts at $\sigma$ and ends
at $\sigma'$. We pull back the family $\star_\sigma$ and the connection
$D_V$ along $\gamma$.
We claim that the parallel transport $\Phi(t)$ of $D_V$ along $\gamma$ exists and is unique.
If we consider a fixed order of $h$, this reduces to the existence and uniqueness
of solutions to an ODE of the form
$$ \frac{d}{dt} F(t) = G(t),$$
with $G(t)$ a given smooth one-parameter family of differential operators on $M$.
Modulo constants (with respect to $t$) there is cleary a unique solution, given
by integration over $t$.

We can now consider the family of star products
$$ \Phi(t)\circ \star_{\sigma} \circ \Phi(t)^{-1}\otimes \Phi(t)^{-1}.$$
At $t=0$, it coincides with $\star_{\gamma(t)}$ and satisfies the same ordinary
differential equation, which is
$$ \frac{d}{dt} \star_{\gamma(t)} = 
\star_{\gamma(t)}\circ (A\left(\frac{d\gamma}{dt} \right)\otimes \mathrm{id} + \mathrm{id}\otimes A\left(\frac{d\gamma}{dt} \right)) - A\left(\frac{d\gamma}{dt} \right) \circ \star_{\gamma(t)}.$$
By uniqueness of the solution to this ODE, we conclude
$$ \star_{\gamma(t)} = \Phi(t)\circ \star_{\sigma}\circ \Phi(t)^{-1}\otimes \Phi(t)^{-1}.$$
In particular, $\star_\sigma$ and $\star_{\sigma'}$ are equivalent star products.
Consequently their characteristic classes coincide.
\end{proof}

The discussion above specializes to the case of a compact symplectic manifold $M$ with a family of compatible K\"ahler structures parametrized by $\T$. As in Subsection \ref{subsection: BT}, one
can consider the family of Berezin-Toeplitz star products $\{\star_\sigma\}_{\sigma \in \T}$.
The characteristic class of $\star_\sigma$ is proportional to the first Chern class of $M$,
 see \cite{Hawkins}.
Since the first Chern class
is independent of the compatible complex structure $\sigma$,
we may apply Theorem \ref{thm:existence_formal_connection_iff_char_class_constant}
to the family $\{\star_\sigma\}_{\sigma \in \T}$ and obtain the existence of a
formal connection which is compatible with $\{\star_\sigma\}_{\sigma \in \T}$. Even better, we might directly apply Theorem \ref{thm:trivializations->trivializations}:
The only input data is a family $\beta_\sigma$ of $1$-forms on $\T$ with values in $\Omega(M)^1[[h]]$
such that the condition from Theorem \ref{thm:trivializations->trivializations} holds.
Using Hodge-theory with respect to the family of K\"ahler metrics yields a preferred such family.
We hence obtain the following result:

\begin{thm}\label{thm: a bit of Kahler}
Let $(M,\omega)$ be a compact, symplectic manifold equipped with a family of compatible K\"ahler structures parametrized by a manifold $\T$. Let us consider the corresponding family of Berezin-Toeplitz star products $\{\star_\sigma\}_{\sigma\in\T}$. Then the family admits a preferred formal connection.
\end{thm}

The details of the proof will be given in future work.
There, we also hope to compare more closely the formal connection from Theorem \ref{thm: a bit of Kahler}
to the expression for the formal Hitchin connection given in Equation (\ref{eq:formal_h_connection_explicit}), Subsection \ref{subsection: formal Hitchin}.

\subsection{Proof of Theorem \ref{thm:trivializations->trivializations}}

This subsection contains the proof of Theorem \ref{thm:trivializations->trivializations}.
We will use Fedosov's framework for the deformation quantization of
symplectic manifolds, which we reviewed in Subsection \ref{subsection: Fedosov}.
The first step of the proof is in fact the reduction to families
of Fedosov star products via \cite[Theorem 4.1]{gutt2003natural}:
Gutt and Rawnsley show there that every natural star product is equivalent
to a preferred star product of Fedosov-type through a preferred equivalence.
Hence we assume from now on without loss of generality that
each member $\star_\sigma$ of the family of star products which we consider
equals $\star_{\nabla_{\sigma},\alpha_{\sigma}}$, where:
\begin{itemize}
\item $\nabla_{\sigma}$ is a symplectic connection for $(M,\omega)$,
\item $\alpha_{\sigma}$ is an element of $\omega + hZ^2(M,\mathbb{R})[[h]]$.
\end{itemize}
Moreover, these data fit together into smooth families parametrized by $\T$.

Our aim is to understand the dependence of Fedosov's construction from Subsection
\ref{subsection: Fedosov} on the data $\{\nabla_\sigma,\alpha_\sigma\}_{\sigma \in \T}$. 
As in the proof of Theorem \ref{thm:trivializations->trivializations}
we can fix a $1$-form $\beta$ on $\T$ with values in $\Omega^1(M)[[h]]$
such that for all vector fields $V$ on $\T$ the equation
$$ d_M i_V \beta = V[\alpha_\sigma]$$
holds.

We now go through Fedosov's construction, seen as fibred over the parameter
space $\T$. In the first step, we realize each star product $\star_{\nabla_\sigma,\alpha_\sigma}$
with the help of a Fedosov connection $D_{r(\sigma)}$. In the following, we interpret
$r$ as a family of elements in $\Omega^2(M,W)$. Recall
that $W$ denotes the Weyl bundle over $M$.
By definition, the star product $\star_{\nabla_\sigma,\alpha_\sigma}$
is given by
$$ f\star_{\nabla_\sigma,\alpha_\sigma} g = p(\tau_{r(\sigma)}(f)\circ_{MW} \tau_{r(\sigma)}(g)),$$
where $\tau_{r(\sigma)}(f)$ is the unique extension of $f$ to $\Gamma(W)$
which is constant with respect to the connection $D_{r(\sigma)}$
and $p$ is the canonical projection $\Gamma(W) \to C_h^\infty(M)$.

We claim that there is a $1$-form $s$ on $\T$ with values in $\Omega^0(M,W)$
such that the operator
$$ \hat{D}_s := d_\tau + \frac{i}{h}\ad(s)$$
commutes with $D_r$ in the graded sense, i.e.
$$ [\hat{D}_s, D_r] := \hat{D}_s \circ D_r + D_r \circ \hat{D}_s = 0.$$

Using the form of $D_r$ and our ansatz for $\hat{D}_s$, we obtain the following
expression for the graded commutator
$$ i_V [\hat{D}_s,D_r] = \frac{i}{h}\ad\left(-D_r(i_Vs) + V[r] \right) + V[d_{\nabla_\sigma}],$$
where $V$ is an arbitrary vector field on $\T$.
Recall that $\nabla_\sigma$ is a family of symplectic connections.
These form an affine space over the subspace $\mathcal{S}$ of $\Omega^1(M,\mathrm{End}(TM))$ which corresponds to totally symmetric contravariant $3$-tensors after contraction with $\omega$.
Hence the variation of $\nabla_\sigma$ is encoded by 
a $1$-form $S_\sigma$ on $\T$ with values in $\mathcal{S}$.
In the following we use $\omega$ to turn $S_\sigma$ into
a $1$-form on $\T$ with values in $\Omega^1(M,W)$.

\begin{lemma}
The variation $V[d_{\nabla_\sigma}]$ of the covariant derivative $d_{\nabla_{\sigma}}$
on $\Omega(M,W)$ can be expressed as
$$ V[d_{\nabla_\sigma}] = \frac{i}{2h} \ad(i_V S_\sigma).$$
\end{lemma}

\begin{proof}
If we consider $a \in W$ which does not depend on $h$, only odd powers of $h$ will appear in the adjoint action
$\ad(a)$ under the Moyal-Weyl product. Hence 
we only need to compute $\ad(i_V S_\sigma)$ to first order, which is essentially
given by contracting $i_V S_\sigma$ by the Poisson bivector field dual to $\omega$,
just undoing the contraction by $\omega$ that occurred before.
\end{proof}

Thanks to the lemma, we arrive at the following formula for the commutator:
$$  i_V [\hat{D}_s,D_r] = \frac{i}{h}\ad\left(-D_r(i_Vs) + V[r] + \frac{1}{2}i_V S_\sigma\right).$$
Our strategy is now the same as in Fedosov's construction:
we have to choose $s$ such that the expression in the bracket
on the right hand side of the above equation lies in the center of $\Omega(M,W)$,
which is $\Omega(M)[[h]]$, i.e. we impose
\begin{equation}\label{eqn:Fedosov-type}
 -D_r(i_Vs) + V[r] + \frac{1}{2}i_V S_\sigma = i_V \beta
 \end{equation}
for $\beta$ a $1$-form on $\T$ with values in $\Omega^1(M)$.
One finds a necessary condition for Equation (\ref{eqn:Fedosov-type}) which reads
$$ V[\alpha_\sigma] = d_M i_V \beta,$$
and coincides with the condition we imposed on $\beta$ at the beginning.

We have the following analogon to Theorem \ref{thm:fedosov_char_class}:

\begin{prop}\label{prop: s}
There is a unique $s \in \Omega^1(\T,\Omega^1(M,W))$ such that
$$  -D_r(i_Vs) + V[r] + \frac{1}{2}i_V S_\sigma = i_V \beta$$
and $\delta^* (i_V s) = 0$ hold for all vector fields $V$ on $\T$.
\end{prop}

\begin{proof}
We rewrite the equation as
$$ D_r(i_Vs) = V[r] + \frac{1}{2}i_V S_\sigma - i_V \beta.$$
By general cohomological considerations,
it suffices to now prove that
the right hand side of the above equation is closed
 with respect to $D_r$.
 
 By $0=[D_r,[d_\T, D_r]] = \pm \frac{i}{h} \ad( D_r(V[r] + \frac{1}{2}i_V S_\sigma - i_V \beta)),$
 we know that $D_r(V[r] + \frac{1}{2}i_V S_\sigma - i_V \beta)$ is central.
 If we compute the component of this element in the center,
 we obtain $V[\alpha_\sigma] - d_M i_V \beta$, which vanishes by assumption.
 
 We now know that there is an appropriate solution $i_Vs$.
 The condition $\delta^*(i_V s)=0$ singles out a unique one. 
\end{proof}

\begin{lemma}\label{lemma: variation of tau}
Let $V$ be a vector field on $\T$.
The variation of $\tau_r(f)$ in the direction of $V$ is given by
$$ V[\tau_r(f)] = \frac{i}{h}(\tau \circ p - \mathrm{id})\left([i_Vs,\tau_r(f)] \right).$$
\end{lemma}

\begin{proof}
Let us define $\mu(f)$ to be $\hat{D}_s\tau_r(f)$.
Since $D_r$ and $\hat{D}_s$ commute, $\mu(f)$ is closed with respect to $D_r$.
Moreover, the image under $p$ computes to
$$ p(\mu(f)) = p( (d_\T + \frac{i}{h}\ad(s))\tau_r(f)) = \frac{i}{h}p([s,\tau_r(f)]).$$
Since $D_r$-closed elements are determined by their image under $p$, we obtain
$$ \mu(f) = \frac{i}{h}(\tau \circ p) ([s,\tau_r(f)]).$$
Inserting $\mu(f) = (d_\T + \frac{i}{h}\ad(s))\tau_r(f)$ into the equality
yields
\[
 d_\T \tau_r(f) = \frac{i}{h}(\tau\circ p -\mathrm{id})([s,\tau_r(f)]). \qedhere 
\]
\end{proof}

\begin{prop}\label{prop: s2}
Let $s \in \Omega^1(\T,\Omega^0(M,W))$ be as in Proposition \ref{prop: s}.

Then
$$ A(V) := \frac{i}{h} p\left([i_Vs,\tau_{r}(\cdot)] \right)$$
defines a $1$-form on $\T$ with values in $\D^1_h(M)$
that satisfies the following:
\begin{enumerate}
\item If $X_f$ denotes the Hamiltonian vector field of $f$ and $\beta_1$ is the component of $\beta$ of order $h$, we have 
$$A(V)(f) = -h (i_V i_{X_f}\beta_1) \quad \pmod{h^2}.$$
\item The identity
 $$d_H i_V A(V) = V[\star_{\nabla_\sigma,\alpha_\sigma}]$$
 holds.
 \end{enumerate}
\end{prop}

\begin{proof}

That $A(V)$ is a differential operator follows
from the fact that the value of  $\tau_r(f)$ at  $x$ depends only on the jet of $f$ at $x$
and that the Moyal-Weyl product only acts fibrewise.

To verify the claim about the lowest order terms of $A(V)$, we have to consider the lowest orders of $s$ and $\tau_r(f)$ with respect to the total degree, which is given by the polynomial degree
in the Weyl-algebra plus twice the power in $h$.
The expansion up to order $1$ of $\tau_r(f)$ is $f + df$,
where $df$ is seen as a function on $\T$ with values in $W$.
Since $f \in C^\infty(M)$ lies in the center of the Weyl-algebra,
only the term $df$ is relevant for our considerations.

Concerning $s$, we notice that we only need to consider its component in
$T^*M \subset W$ because we are only interested in $ p([s,df]),$
and all other components lead to terms that project to zero under $p$.
Inspecting this component in lowest order, we find an element $a$
which is uniquely determined by
$$ -\delta(i_Va) = -hi_V\beta_1, \quad \textrm{and } \delta^*a=0,$$
where $\beta_1$ is the term of order $h$ in $\beta$.
The solution to this equation is given by
$$ a = h \delta^* i_V\beta_1.$$
In total, we obtain that the lowest order term of $A(V)$
is given by
$$ \frac{i}{h} \ad(a)(df) = i [\delta^*i_V \beta_1,df] = -h (i_V i_{X_f}\beta_1),$$
where $X_f$ denotes the Hamiltonian vector field of $f$.

It remains to verify that the image of $A(V)$ under the Hochschild differential
$d_H$ is $V[\star_\sigma]$.
In fact we compute:
\[
V[f\star_{\nabla_\sigma,\alpha_\sigma} g] = p(V[\tau_r(f)] \circ_{MW} \tau_r(g))
+ p(\tau_r(f) \circ_{MW} V[\tau_r(g)]).
\]  
We now use the expression for the variation of $\tau_r(f)$ and $\tau_r(g)$
we obtained in Lemma \ref{lemma: variation of tau} and arrive at
$$
V[f\star_{\nabla_\sigma,\alpha_\sigma} g] =  A(V)(f) \star_{\nabla_\sigma,\alpha_\sigma} g - f \star_{\nabla_\sigma,\alpha_\sigma} A(V)(g) - A(V)(f\star_{\nabla_\sigma,\alpha_\sigma} g),
$$
which is exactly $V[\star_{\nabla_\sigma,\alpha_\sigma}] = d_H A(V)$.
\end{proof} 

Combining Proposition \ref{prop: s} and \ref{prop: s2}, we obtain an assignment
 $$C: \mathcal{C} \to \mathcal{P}$$ and thereby complete the proof of Theorem \ref{thm:trivializations->trivializations}.

The following result expresses the curvature of the formal connection
$D_V = V + A(V)$ in terms of the element $s$:

\begin{prop}
Let $A(V):= \frac{i}{h}p([i_Vs,\tau_r(\cdot)])$ be the connection
$1$-form associated to the element $s$ from
Proposition \ref{prop: s}.
The curvature of the formal connection $D_V = V + A(V)$
equals the $2$-form $\Omega_s$ on $\T$ with values in formal differential operators
on $M$ given by
$$ \Omega_s(f) := \frac{i}{h}p([d_\T s + \frac{i}{h} s\circ_{MW} s, \tau_r(f)]).$$
\end{prop}

\begin{proof}
We first compute $(d_\T A)(f) = d_\T\frac{i}{h}p([s,\tau_r(f)])$.
By Lemma \ref{lemma: variation of tau}, we obtain
$$ \frac{i}{h}p([d_\T s,\tau_r(f)] - [s,\frac{i}{h}(\tau\circ p - \mathrm{id})[s,\tau_r(f)]]).$$
On the other hand, applying $A \wedge A$ to $f$ yields
$$ \left(\frac{i}{h}\right)^2p([s,\tau_r p([s,\tau_r(f)])]),$$
which cancels with one of the terms from $(d_\T A)(f)$.
The remaining terms are
$$ \frac{i}{h}p([d_\T s,\tau_r(f)] + \frac{i}{h}[s,[s,\tau_r(f)]]).$$
In order to arrive at the claimed expression, we apply the identity $[s,[s,X]] = \frac{1}{2}[[s,s],X]= [s\circ_{MW}s,X]$.
\end{proof}

\addcontentsline{toc}{section}{Bibliography}
\bibliographystyle{amsalpha}
\bibliography{biblio}

\end{document}